\documentclass[11pt,english]{amsart}

\usepackage{amsmath,amssymb}
\usepackage{color}
\usepackage{mathrsfs}
\usepackage{graphicx}

\numberwithin{equation}{section}
\newtheorem{theorem}{Theorem}[section]
\newtheorem{lemma}[theorem]{Lemma}

\newtheorem{proposition}[theorem]{Proposition}

{\theoremstyle{definition}
\newtheorem{definition}{Definition}[section]}

\def\R{ \mathbb R}
\def\H{H^\infty}

\def\D{{ \mathbb D}}
\def\C{{ \mathbb C}}
\def\Z{{ \mathbb Z}}
\def\N{{ \mathbb N}}
\def\mN{{ \mathbb N}}

\def\mC{{ \mathbb C}}
\def\e{\varepsilon}

\def\union{\cup}
\def\Union{\bigcup}
\def\inter{\cap}
\def\Inter{\bigcap }
\def\ov{\overline}

\def\ss{\subseteq}

\def\buildrel#1_#2^#3{\mathrel{\mathop{\kern 0pt#1}\limits_{#2}^{#3}}}

\def\BP{Blaschke product}

\begin{document}

\title[Topological stable rank of $H^\infty(\Omega)$]{Topological stable
  rank of $H^\infty(\Omega)$ for circular domains $\Omega$}


\author{Raymond Mortini}
\address{\small D\'{e}partement de Math\'{e}matiques\\
\small LMAM,  UMR 7122,
\small Universit\'{e} Paul Verlaine\\
\small Ile du Saulcy\\
\small F-57045 Metz, France}

\email{mortini@poncelet.univ-metz.fr}

\author {Rudolf Rupp}
\address{\small Fakult\"at Allgemeinwissenschaften \\
\small Georg-Simon-Ohm-Hochschule N\"urnberg\\
\small Kesslerplatz 12, D-90489 N\"urnberg, Germany}

\email{rudolf.rupp@ohm-hochschule.de}

\author{Amol Sasane}
\address{\small Department of Mathematics\\
\small London School of Economics\\
\small Houghton Street, London WC2A 2AE, United Kingdom}

\email{A.J.Sasane@lse.ac.uk}

\author[B. D. Wick]{Brett D. Wick$^*$}
\address{School of Mathematics\\ Georgia Institute of Technology\\ 686 Cherry Street\\ Atlanta, GA 30332-0160 USA}
\email{wick@math.gatech.edu}
\thanks{$*$ Research supported in part by a National Science Foundation DMS Grant \# 0752703.}


\begin{abstract}
  Let $\Omega$ be a circular domain, that is, an open disk with
  finitely many closed disjoint disks removed. Denote by
  $H^\infty(\Omega)$ the Banach algebra of all bounded  holomorphic
  functions on $\Omega$, with pointwise operations and the supremum
  norm.  We show that the topological stable rank of
  $H^\infty(\Omega)$ is equal to $2$. The proof is based on Suarez's
  theorem that the topological stable rank of $H^\infty(\D)$ is equal
  to $2$, where $\D$ is the unit disk. We
  also show that for domains symmetric to the real axis, the Bass and
  topological stable ranks of the real symmetric algebra
  $\H_\R(\Omega)$ are 2.
\end{abstract}

\maketitle

\centerline {\small\the\day.\the \month.\the\year} \medskip

\section{Introduction}

The aim of this short note is to prove that the topological stable
rank of the Banach algebra $H^\infty(\Omega)$ of all bounded analytic
functions on $\Omega$ is equal to $2$, where $\Omega$ denotes a
circular domain. By conformal equivalence, the same assertion will
hold for any finitely connected, proper domain in $\C$ whose boundary does not
contain any one-point components.  We shall also show that for circular domains $\Omega$
that are symmetric to the real axis, the real algebra
$$
\H_\R(\Omega)=\{f\in H^\infty(\Omega)\; : \; (f(z^*))^*= f(z) \;(z\in \Omega)\}
$$
has the Bass and topological stable rank $2$. Here $z^*$ denotes the
complex conjugate of $z$. The precise definitions are given below.

The notion of the topological stable rank of a Banach algebra was
introduced by M. Rieffel in \cite{Rie}, in analogy with the notion of
the (Bass) stable rank of a ring defined by H. Bass \cite{Bas64}. We
recall these definitions now.

\begin{definition}
  Let $R$ be a commutative ring with identity element $1$.  An
  $n$-tuple $a:=(a_1,\dots, a_n)\in R^n$ is said to be {\sl
    invertible} or {\sl unimodular}, (for short $a\in U_n(R)$), if
  there exists a solution $(x_1,\dots, x_n)\in R^n$ of the Bezout
  equation $\sum_{j=1}^n a_jx_j=1$. We say that $a=(a_{1},\dots,
  a_{n},a_{n+1})\in U_{n+1}(R)$ is {\em reducible} if there exist
  $h_{1}, \dots, h_{n} \in R$ such that $ (a_{1}+h_{1}a_{n+1}, \dots,
  a_{n}+h_{n}a_{n+1})\in U_{n}(R)$.

The {\em Bass stable rank of} $R$ (denoted by $\text{\em bsr}\; R$)
 is the least $n\in \mN$ such that
every element $a=(a_{1},\dots, a_{n},a_{n+1}) \in U_{n+1}(R)$ is
reducible, and it is infinite if no such integer $n$ exists.

Let $A$ be a commutative Banach algebra with unit element $1$. The
least integer $n$ for which $U_n(A)$ is dense in $A^n$ is called the
{\sl topological stable rank} of $A$ (denoted by $\text{\em tsr}\; A$)
and we define $\textrm{\em tsr} \;A= \infty$ if no such integer $n$
exists.  

It is well known that $\text{bsr}\, A\leq \text{tsr}\, A$; see
\cite[Corollary 2.4]{Rie}.
\end{definition}

In the case of the classical algebra $H^\infty(\D)$ of the unit disk $\D=\{z\in\C: |z|<1\}$, 
D. Suarez \cite{su} showed that the topological stable rank is $2$. We
will use this result in order to derive our result for
$H^\infty(\Omega)$ when $\Omega$ is a circular domain.

\begin{theorem}[Suarez \cite{su}]
\label{thm_su}
The topological stable rank of $H^\infty(\D)$ is $2$. 
\end{theorem}

Let us recall that previously Tolokonnikov \cite{to} showed that
the Bass stable rank of $H^\infty(\Omega)$ is $1$. That was based on
S. Treil's \cite{tr} fundamental result that $\H(\D)$ has the Bass
stable rank $1$.

In \cite{mw} Mortini and Wick showed that the Bass and topological stable ranks
of the real symmetric algebra
$$
\H_\R(\D)=\{f\in H^\infty(\D): \;(f(z^*))^*= f(z)\; (z\in \D)\}
$$
are $2$.  Using this
we will show that $\D$ can be replaced by an arbitrary circular domain
symmetric to the real axis.

We now give the precise definition of a circular domain, and also fix
some convenient notation.

\medskip 

\noindent {\bf \em Notation.} Let $\Omega$ be a {\em circular domain},
of connectivity $n$, that is, an open disk, $D$, with $n-1$ closed
disjoint disks removed \footnote{We tacitly assume that the closures
  of the removed disks are contained within $D$.}. Then $\Omega$ is the
intersection of $n$ simply connected domains, $ \Omega=\Omega_0 \cap
\Omega_1 \cap \cdots \cap \Omega_{n-1}$, where $\Omega_i=\overline{\C}
\setminus \overline{D_i}$, the $D_i$ being open disks in the
extended complex plane $\overline{\C}=\C \cup \{\infty\}$.  We assume
that $\infty\in D_0$.  The boundary of a set $\Omega \subset \mC$ is
denoted by $\partial \Omega$.  

Let $H(\Omega)$ denote the set of all holomorphic functions on
$\Omega$, and let $H^\infty(\Omega)$ be the Banach algebra of all
bounded  holomorphic functions on $\Omega$, with pointwise
operations and the supremum norm.

If $\Omega$ is real symmetric (that is, $z \in \Omega$ if and only if
$z^{*} \in \Omega$), then we use the symbol $\H_{\R}(\Omega)$ to
denote the set of functions $f$ belonging to $\H(\Omega)$ that are
{\em real symmetric}, that is, $f(z)=(f(z^{*}))^{*}$ ($z\in \Omega$).

\medskip 

An example of a circular domain is the annulus $\mathbb{A}=\{z\in
\C:\;r_1<|z|<r_2\}$, where $0<r_1<r_2$. In this case
$\mathbb{A}= \Omega_0 \cap \Omega_1$, where 
\begin{eqnarray*}
  \Omega_0&:=& \{ z\in \C:\; |z|<r_2\},\\
  \Omega_1&:=& \{ z\in \overline{\C} :\; |z|>r_1\}.
\end{eqnarray*}
Thus $\Omega_0=\overline{\C} \setminus\overline{D_0}$ and
$\Omega_1=\overline{\C} \setminus\overline{D_1}$, where
\begin{eqnarray*}
  D_0&:=& \{ z\in \overline{\C}:\; |z|> r_2\},\\
  D_1&:=& \{ z\in {\C}: \; |z|< r_1\}.
\end{eqnarray*}

Our main results are the following:

\begin{theorem}
\label{main_thm}
Let $\Omega$ be a circular domain. The topological stable rank of
$H^\infty(\Omega)$ is $2$.
\end{theorem}

\begin{theorem} 
\label{sym} 
Let $\Omega$ be circular domain symmetric to the real axis. Then the
topological and Bass stable rank of $\H_\R(\Omega)$ is $2$.
\end{theorem}

\section{Preliminaries}

The following Cauchy decomposition is well known (for $H^p(\Omega)$
functions, $1\leq p \leq \infty$) \cite[Proposition 4.1, p. 86]{fi} or
\cite[Theorem 10.12, p.181]{du}.  

\begin{lemma}
\label{cauchy}
Let $\Omega=\Inter_{j=0}^{n-1} \Omega_j$ be a circular domain of
connectivity $n$. Then any $f\in H(\Omega)$ can be decomposed as
$f=f_0+f_1+\cdots + f_{n-1}$, where $f_j\in H(\Omega_j)$. If
additionally the real part of $f$ is bounded above on $\Omega$, then
the same is true for the $f_j$.
\end{lemma}
\begin{proof}
  Apply Cauchy's integral formula for a null homologic cycle, close to
  the boundary of $\Omega$, and use the principle of analytic
  continuation. Now let us assume that the real part of $f$ is bounded
  above on $\Omega$.  Fix $k\in\{0,1,\dots, n-1\}$. Since
  $f_j(\infty)=0$ for $j=1,2,\dots, n-1$ and $\sum_{j\not=k} f_j$ is
  holomorphic in a neighborhood of the set $\ov \C\setminus \Omega_k$,
  we see that the real part of each $f_j$ is bounded above on
  $\Omega_j$, for $j=0, 1,\dots,n-1$.  
\end{proof}

We will use the following factorization result; the non-symmetric
version appears in \cite[Lemma 1]{to}. Since in our viewpoint, the
proof of the annulus-case by Tolokonnikov is not complete, we give a
more general proof, that includes also the symmetric case.

Recall that a \BP\ $B$ with zeros $(z_j)$ in the disk 
$$
D(a,r)=\{z\in \C:\; |z-a|<r\}
$$ 
has the form $B(z)= b(\frac{z-a}{r})$, where $b$ is the usual \BP\ of
the unit disk with zeros $w_j=\frac{z_j-a}{r}$.  Similarly, the \BP\
$B_e$ with zeros $(z_j)$ in the exterior of the disk $D(a,r)$ has the
form $B_e(z)= b(\frac{r}{z-a})$ where $b$ is the usual \BP\ of the
unit disk with zeros $w_j=\frac{r}{z_j-a}$. We call these functions
{\sl generalized \BP s.}

\begin{proposition}
\label{to_fact}
Let $\Omega$ be a circular domain of connectivity $n$, $n\in \N$, and
let $\overline{D_j}$ denote the bounded components of $\ov\C\setminus
\Omega$, $(j=1,\cdots, n-1)$, that is, $D_j$ is the open disk $D(a_j,
r_j)$.  Define
\begin{eqnarray*}
\Omega_j&=&\ov \C\setminus \ov D_j,\quad  j=1,\dots, n-1,\\
\Omega_0&=&\Omega\union \bigg(\bigcup_{j=1}^{n-1} D_j\bigg).
\end{eqnarray*}  
Then every function $f$ in $H^\infty(\Omega)$, $f\not\equiv 0$, can be decomposed as:
$$
f=f_0 \cdot f_1 \cdot f_2  \cdots  f_{n-1}\cdot r, 
$$
where 
$$
f_j \in H^\infty(\Omega_j) \cap \Bigg(H^\infty \Bigl(\bigcup_{k\neq
  j}D_k\Bigr)\Bigg)^{-1}, \quad j=0,1,2, \dots, n-1,
$$ 
and where $r$ is a rational function with poles and zeros contained in
the set $\{a_1,\cdots, a_{n-1}\}.$

If $\Omega$ is a domain symmetric to the real axis, and $f\in
\H_\R(\Omega)$, then each of the functions $f_j$ and $r$ above can be
taken to be real symmetric themselves.
\end{proposition}
\begin{proof}
  We may
  assume that $\Omega$ is the circular domain
  $$
  \Omega=D(a_0,r_0)\setminus \Union_{j=1}^{n-1} \ov{D(a_j, r_j)},
  $$
  where  $\ov D_j=\ov{D(a_j, r_j)} \ss D(a_0,r_0)$
  and where the closures of the $D_j$ $(j=1,\dots,n-1)$ are disjoint.
  
  Let $D_0:= D(a_0,r_0)$.  Set $\Omega_j=\ov\C\setminus\ov D_j$,
  $(j=0,1,\dots,n-1)$.  It is well known that the sequence $(z_k)$ of
  zeros of $f$ satisfies the generalized Blaschke condition; that is
  $\sum_k\text{dist}\; (z_k,\partial \Omega)$ converges (see \cite{fi,
    ru}). Split $(z_k)$ into $n$ sequences $(z_{k,j})_k$, $j=0,1,\dots
  n-1$, so that the cluster points of $(z_{k,j})_k$ are exactly those
  of $(z_k)$ that belong to $\partial D_j$,
  $j=0,1,\dots,n-1$. Let $B_j$ be the generalized \BP\ formed with the
  zeros $(z_{k,j})_k$ of $f$, $j=0,1,\dots, n-1$.  It is clear that
  the zeros of $B_j$ cluster only at $\partial D_j$, $0\leq j\leq
  n-1$.  
  
  Then $f$ can be written as $f=B_0\cdot B_1\cdots B_{n-1}\cdot g$,
  where $g\in \H(\Omega)$ and $g$ has no zeros in $\Omega$ (note that
  here we have used the fact that divison by $B_j$ does not change the
  relative supremum of $f$ on the boundary of $\Omega_j$).

  By \cite[p. 111-112]{bu}, there exist $k_j\in \Z$ and $h$
  holomorphic in $\Omega$, such that
  $$
  g(z)=\prod_{j=1}^{n-1}(z-a_j)^{k_j}e^ {h(z)}.
  $$  
  Note that the real part of $h$ is bounded above on $\Omega$.
  
  By Lemma \ref{cauchy}, there exist $h_j\in H(\Omega_j)$ such that
  $h=h_0+h_1+\cdots +h_{n-1}$ and the real part of each $h_j$ is
  bounded above on $\Omega_j$, for $j=0, 1,\cdots,n-1$. Hence the
  functions $e^{h_j}\in \H(\Omega_j)$.
 
  Now $f\!\!=\!r\!\displaystyle\prod_{j=0}^{n-1}\!\!B_j e^{h_j}$,
  where $r(z)\!\!=\!\!\displaystyle \prod_{j=1}^{n-1}\!(z-a_j)^{k_j}$
  gives the desired factorization.

  In case of a symmetric domain $\Omega$ and $f\in \H_\R(\Omega)$,
   we can choose $a_j$ to be real if the
  disk $D(a_j,r_j)$ meets the real line, and the other $a_j$ in pairs
  $(a,a^*)$. Thus we can ensure that $r$ is real symmetric, because
  the exponents $k_j$ are the same for $a_j$ and $a_j^*$ due to the
  fact that
$$
k_j = \frac{1}{2 \pi i}\oint_{\Gamma}\frac{g'(z)}{g(z)}dz \;,
$$
where $\Gamma$ denotes a suitable small circle around $a_j$.  

The Blaschke products above are easily seen to be choosable in a real
symmetric fashion.  Hence, since $f$ is real symmetric, we
conclude that $g$ is real symmetric as well. Therefore, $e^h$ is real
symmetric; that is 
$$
e^{h(z)}=(e^{h(z^*)})^*=e^{(h(z^*))^*}.
$$
Since $\Omega$ is a domain, $h(z)-(h(z^*))^*$ equals a constant $2k\pi
i$ for some $k\in\Z$. Therefore
$$
h(z)=\frac{h(z)+(h(z^*))^*}{2}+\frac{h(z)-(h(z^*))^*}{2}
=\frac{h(z)+(h(z^*))^*}{2}+k \pi i
$$
Now in Cauchy's decomposition, we simply consider the symmetric
functions $H_j(z):=\frac{h_j(z)+(h_j(z^*))^*}{2}$, and derive
$$
h(z)=\sum_{j=0}^{n-1}H_j(z)+k \pi i.
$$
Thus we have one of the following cases
$$
e^{h(z)}=e^{\sum_{j=0}^{n-1}H_j(z)} \quad(z \in \Omega)
$$
or
$$
e^{h(z)}=-e^{\sum_{j=0}^{n-1}H_j(z)} \quad(z \in \Omega).
$$
In the latter case we take $-r$ instead of $r$.  Thus all the factors
in
  $$
  f=r\displaystyle \prod_{j=0}^{n-1}B_je^{H_j}
  $$ 
  are symmetric.
\end{proof}
 
We recall that the corona theorem holds for $H^\infty(\Omega)$ when
$\Omega$ is a circular domain; see for example \cite[Theorem 6.1,
p.195]{fi}.

\begin{proposition} 
Let $\Omega$ be a circular domain. Then $(f_1,\dots,f_n)$ is 
invertible in $H^\infty(\Omega)$ if and only if there exists a 
$\delta>0$ such that
$$
\sum_{j=1}^n |f_j(z)| \geq \delta \quad (z\in \Omega).
$$
\end{proposition}

This corona-theorem is of course true for $\H_\R(\Omega)$.  Indeed, if
$f_j\in\H_\R(\Omega)$ and $(g_1,\dots,g_n)$ is a solution of
$\sum_{j=1}^n g_jf_j=1$ in $\H(\Omega)$, then $(\widetilde
g_1,\dots, \widetilde g_n)$ is a solution of the Bezout equation
$\sum_{j=1}^n \widetilde g_j f_j=1$ in $\H_\R(\Omega)$, where $\widetilde g_j(z):=
\frac{g(z)+(g_j(z^*))^*}{2}$ ($z \in \Omega$).  

\bigskip

We will need two technical results, which are proved below. In the
following, the notation $M(R)$ is used to denote the maximal ideal
space of the unital commutative Banach algebra $R$. Also the complex
homomorphism from $H^\infty(\Omega)$ to $\mC$ of point evaluation at a
point $z \in \Omega$ will be denoted by $\varphi_{z}$, that is,
$\varphi_{z} (f)=f(z)$, $f\in H^\infty(\Omega)$.

Let $z_0\in \overline \Omega$. The set
$$
M_{z_0}(\H(\Omega))=\{\varphi\in M(H^\infty(\Omega)): \varphi(z)=z_0\}
$$
is called the {\em fiber of $M(H^\infty(\Omega))$ over $z_0$}. It is
well known (see \cite{fi}), that we have $\varphi(f)=0$ for some 
$\varphi\in M_{z_0}(\H(\Omega))$ if
and only if $\liminf_{z\to z_0} |f(z)|=0$.  The {\em zero set of $f\in
\H(\Omega)$} is the set $\{\varphi\in M(\H(\Omega)): \varphi(f)=0\}$.

We need a Lemma that lets us decompose two functions that live on different circular domains. To this end, let $D_1, D_2$ be open disks in $\ov\C$ such that $\overline{D_1}
\cap \overline{D_2} =\emptyset$.  Next, define $\Omega_j:=\overline{\mC}\setminus\overline{D_j}$ for $j=1,2$.  Suppose that $f_j\in H^\infty(\Omega_j)$ for $j=1, 2$ are non-zero functions.  Next, set
\begin{eqnarray*}
Z_1&=&\bigg\{\xi\in \partial D_1=\partial \Omega_1:  f_2(\xi)=0~~ \text{ and }~~ 
  \liminf_{z\to \xi \atop z\in \Omega_1\inter \Omega_2}|f_1(z)|=0\bigg\},\\
Z_2&=&\bigg\{\xi\in \partial D_2=\partial \Omega_2:  f_1(\xi)=0~~ \text{ and }~~ 
  \liminf_{z\to \xi \atop z\in \Omega_1\inter \Omega_2}|f_2(z)|=0\bigg\},
\textrm{ and}\\
Z_3&=&\bigg\{a\in \Omega_1\inter \Omega_2: f_1(a)=f_2(a)=0\bigg\},
\end{eqnarray*}

\begin{lemma} 
\label{lemma_1}
Let $D_1, D_2$ be open disks in $\ov\C$ such that $\overline{D_1}
\cap \overline{D_2} =\emptyset$. Define $\Omega_j:=
\overline{\mC}\setminus \overline{D_j}$.  Let $f_j \in
H^\infty(\Omega_j)$ be nonzero functions.
Then the zero sets of $f_1$ and $f_2$ meet in at most a finite number of
fibers of $\H(\Omega_1\inter \Omega_2)$. In other words, there exist
at most finitely many $z_j\in\ov{\Omega_1\inter \Omega_2}$ for which
$$
\liminf_{z\to z_j} |f_1(z)|=\liminf _{z\to z_j}|f_2(z)|=0.
$$
Moreover,   
$f_1$ and $f_2$ can be written as 
\begin{eqnarray*}
f_1&=&\prod_{z_j\in Z_2\union Z_3}  (z-z_j)^{m_j} \widetilde F_1, \textrm{ and}\\
f_2&=&\prod_{z'_j\in Z_1\union Z_3}  (z-z'_j)^{m_j'} \widetilde F_2
\end{eqnarray*}
where $\widetilde F_j$ is in $\H(\Omega_1\inter
\Omega_2)$ and has the property that for any element $\varphi \in
M(\H(\Omega_1\inter \Omega_2))$ either $\varphi(\widetilde F_1)\not=0$
or $\varphi(\widetilde F_2)\not=0$.

Additionally, when $\lambda \in\ov{\Omega_1\inter \Omega_2}$, each
$\varphi\in M_\lambda( \H(\Omega_1\inter \Omega_2))$ is such that
$\varphi=\varphi_{\lambda} \in
M(H^\infty(\Omega_1))$ whenever $\lambda \in \Omega_2$, or 
 $\varphi=\varphi_{\lambda} \in M(H^\infty(\Omega_2))$
whenever $\lambda \in \Omega_1$.

\end{lemma}
\begin{proof} It is clear that if $\varphi \in M(H^\infty(\Omega_1
  \cap \Omega_2))$, then $\varphi \in M(H^\infty (\Omega_1))$ and
  $\varphi \in M(H^\infty(\Omega_2))$.

  Now the set $ Z_3=\{z\in \Omega_1 \cap \Omega_2\;|\; f_1(z)=f_2(z)=0\}$
  is finite, for otherwise, there is an accumulation point of zeros in
  $\partial \Omega_1$ or in $\partial \Omega_2$. But $\partial
  \Omega_1$ is contained in $\Omega_2$, and $\partial \Omega_2$ is
  contained in $\Omega_1$. So either $f_1$ or $f_2$ is identically $0$, a
  contradiction.

  Consider the set $Z_2$ and let $\lambda\in Z_2$. 
 There are only finitely many zeros of $f_1$ on the circle $\partial D_2
\subset \Omega_1$, since  $f_1$ is not identically zero.
 Similarly, we can can argue in the case when $\lambda \in Z_1$.
 Thus, $Z_1$ is finite as well. This completes the proof.
\end{proof}

It is clear that an analogous version holds true for the symmetric
case.

\begin{lemma}
\label{lemma_2}
Let $D_1, D_2$ be open disks in $\overline{\C}$ such that
$\overline{D_1} \cap \overline{D_2} =\emptyset$. Define $\Omega_1:=
\overline{\mC}\setminus \overline{D_1}$ and $\Omega_2:=
\overline{\mC} \setminus \overline{D_2}$.  Let $f_1,g_1 \in
H^\infty(\Omega_1)$ and $f_2,g_2\in H^\infty(\Omega_2)$ be 
nonconstant functions such that there exists $\delta >0$ such that the
following hold:
\begin{itemize}
\item[(P1)] For all $z\in \Omega_1$, $|f_1(z)|+|g_1(z)|\geq \delta$.
\item[(P2)] For all $z\in \Omega_2$, $|f_2(z)|+|g_2(z)|\geq \delta$.
\end{itemize}
Then, for every $\e>0$,  there exist $F_1, G_1\in \H(\Omega_1), F_2,G_2\in \H(\Omega_2)$ such that
\begin{itemize}
\item[(C1)] $(F_1,G_2)$ is invertible in $\H(\Omega_1\inter \Omega_2)$,
\item[(C2)] $(G_1,F_2)$ is invertible in $\H(\Omega_1\inter \Omega_2)$
\item[(C3)] $(F_1,G_1)$ is invertible in $\H(\Omega_1)$,
\item[(C4)] $(F_2,G_2)$ is invertible in $\H(\Omega_2)$, and 
\item[(C5)] $||f_1-F_1||+||g_1-G_1||+||f_2-F_2||+||g_2-G_2||<\e$.
\end{itemize}
In particular, $(F_1F_2, G_1G_2)$ is invertible  in $\H(\Omega_1\inter \Omega_2)$.
\end{lemma}

\begin{proof} 
  Consider the pair $(f_1,g_2)\in \H(\Omega_1)\times \H(\Omega_2)$. By
  Lemma \ref{lemma_1} we may perturb the finitely many zeros of $f_1$
  belonging to $S_2\union S_3$ and those of $g_2$ that lie in $S_1$ so
  that the new functions $F_1$ and $G_2$ form an invertible pair in
  $\H(\Omega_1\inter \Omega_2)$. Now we do the same with the pair
  $(g_1,f_2)$ in $\H(\Omega_1)\times \H(\Omega_2)$. This gives an
  invertible pair $(G_1,F_2)\in \H(\Omega_1\inter \Omega_2)$. By
  choosing these perturbations sufficiently small, we see that the
  pairs $(F_1,G_1)$ and $(F_2,G_2)$ stay invertible in the associated
  space $\H(\Omega_1)$, respectively $\H(\Omega_2)$.  This yields that
  $(F_1F_2,G_1G_2)$ is invertible in $\H(\Omega_1\inter \Omega_2)$.
\end{proof}

It is clear that an analogous version holds true for the symmetric
  case.

\section{Proof of tsr($H^\infty(\Omega)$)$=2$}

\begin{proof}[Proof of Theorem \ref{main_thm}] 
  Let $f, g \in H^\infty(\Omega)$. By Proposition \ref{to_fact}, we can
  write
\begin{eqnarray*}
f&=& f_0 \cdot f_1 \cdot \dots \cdot f_{n-1} \cdot r, \\
g&=& g_0 \cdot g_1 \cdot \dots \cdot g_{n-1} \cdot s.
\end{eqnarray*} 
where $f_j$ and $g_j\in \H(\Omega_j)$. We note that since the rational
functions $r, s$ have zeros and poles only in the set $\{a_1, \dots,
a_{n-1}\}$, it follows that $r, s$ are invertible in
$H^\infty(\Omega)$.  Since each $\Omega_i$ is simply connected, it
follows from the fact that the topological stable rank of
$H^\infty(\D)$ is $2$ and the Riemann mapping theorem, that also the
topological stable rank of $H^\infty(\Omega_i)$ is equal to $2$. Hence
the pairs $(f_0,g_0), \dots, (f_{n-1},g_{n-1})$ can be replaced by
unimodular pairs $(\widetilde{f}_0, \widetilde{g}_0), \dots,
(\widetilde{f}_{n-1}, \widetilde{g}_{n-1})$ such that for every
$i=0,1, \dots, n-1$
$$
\|f_i- \widetilde{f}_i\|_\infty+ \|g_i -\widetilde{g}_i\|_\infty <\epsilon.
$$ 
By a repeated application of Lemma \ref{lemma_2} to the pairs
$(\widetilde{f}_k, \widetilde{g}_j)$ with $j\neq k$, we get the existence of 
$F_0, \dots, F_{n-1}$, $G_0, \dots,G_{n-1}$, such that 
$$
\|F_k-f_k\|_\infty+\|G_k-g_k\|_\infty<\epsilon,
$$
and the pair $(F_k, G_j)$ is unimodular in $H^\infty (\Omega_k \cap \Omega_j)$
for all $0\leq k,j\leq n-1$. By the elementary theory of Banach
algebras, it follows that there exists a $\delta>0$ such that 
$$
|F_k(z) |+|G_j(z)| \geq \delta \quad (z\in \Omega_k \cap \Omega_j).
$$
Thus there exists a $\delta'>0$ such that with 
\begin{eqnarray*}
\widetilde{f}&:=&F_0\cdot F_1 \cdot \dots \cdot F_{n-1} \cdot r,\\
\widetilde{g}&:=& G_0 \cdot G_1 \cdot \dots \cdot G_{n-1} \cdot s,
\end{eqnarray*} 
we have for all $z\in \Omega=\Omega_0\cap \dots \cap \Omega_{n-1}$, 
$$
|\widetilde{f}(z)|+|\widetilde{g}(z)|\geq \delta'.
$$
By the corona theorem for $H^\infty(\Omega)$, we obtain that
$(\widetilde{f}, \widetilde{g})$ is a unimodular pair in
$H^\infty(\Omega)$. Also, it can be seen that given $\epsilon'>0$, we can
choose $\epsilon>0$ small enough at the outset so that
$$
\|f-\widetilde{f}\|_\infty +\|f-\widetilde{g}\|_\infty 
\leq \epsilon'.
$$
This completes the proof.
\end{proof}

The same proof shows that the topological stable rank of
$\H_\R(\Omega)$ is $2$ as well.  Since the unimodular pair $(z,
1-z^2)$ is not reducible (here we assume that $]-1,1[\;\ss\Omega$,
$-1,1\notin\Omega$,) we have that the Bass stable rank of
$\H_\R(\Omega)$ is not one.  Since the Bass stable rank is always less than the
topological stable rank, we obtain that it must be $2$.

\end{document}